\documentclass[10pt,reqno]{amsart}
\usepackage{hyperref,amsmath,amsfonts,amssymb,amsthm,mathscinet,enumerate}
\usepackage[mathcal]{euscript}
\usepackage[numbers]{natbib}
\usepackage[a5paper, left=8mm, right=8mm, top=1.5cm, bottom =1cm]{geometry}
\usepackage{marginfix}
\usepackage{tikz}
\usetikzlibrary{lindenmayersystems}
\usetikzlibrary{calc}

\def\R{\mathbb{R}}
\def\N{\mathbb{N}}
\def\Z{\mathbb{Z}}
\def\d{{\fam0 d}}

\def\codim{\operatorname{codim}}
\def\dim{\operatorname{dim}}

\def\rank{\operatorname{rank}}
\def\diam{\operatorname{diam}}
\def\Id{\operatorname{Id}}
\def\dist{\operatorname{dist}}
\def\embed{\hookrightarrow}

\newtheorem{theorem}{Theorem}[section]
\newtheorem{lemma}[theorem]{Lemma}
\newtheorem{corollary}[theorem]{Corollary}
\newtheorem{proposition}[theorem]{Proposition}
\newtheorem{remark}[theorem]{Remark}

\numberwithin{equation}{section}
\expandafter\let\expandafter\oldproof\csname\string\proof\endcsname
\let\oldendproof\endproof
\renewenvironment{proof}[1][\proofname]{%
  \oldproof[\bf #1]%
}{\oldendproof}
\def\Li{{L^1}}

\def\Linf{{L^\infty}}
\def\C{{C}}

\def\Lnek{{L^\infty}}
\def\Lio{{L^1_0}}
\def\lin{{\ell^1_n}}
\def\linfn{{\ell^\infty_n}}
\def\Ldi{{L^{d,1}}}

\def\Id{\mathop{\mathrm{Id}}}
\def\vx{\mathbf{x}}
\def\vy{\mathbf{y}}

\include{widebar}
\let\bar\widebar

\begin{document}

\title[]{Strict s-numbers of non-compact Sobolev embeddings into continuous functions}

\begin{abstract}
For limiting non-compact Sobolev embeddings into continuous functions we study
behavior of Approximation, Gelfand, Kolmogorov, Bernstein and Isomorphism
$s$-numbers. In the one dimensional case the exact values of the above-mentioned
strict $s$-numbers were obtained and in the  higher dimensions sharp estimates
for asymptotic behavior of strict $s$-numbers were established. As all known
results for $s$-numbers of Sobolev type embeddings are studied mainly under
the compactness assumption then our work is an extension of existing results
and reveal an interesting behavior of $s$-numbers in the limiting case when
some of them (Approximation, Gelfand and Kolmogorov) have positive lower
bound and others (Bernstein and Isomorphism) are decreasing to zero. From our
results also follows that such limiting non-compact Sobolev embeddings
are finitely strictly singular maps.
\end{abstract}

\author[J. Lang]{Jan Lang}
\address{%
Department of Mathematics,
Ohio State University,
Columbus OH,
43210-1174 USA}
\email{lang@math.osu.edu}
\urladdr{0000-0003-1582-7273}

\author[V. Musil]{V\'\i t Musil}
\address{%
Department of Mathematics,
Ohio State University,
Columbus OH,
43210-1174 USA
---
Department of Mathematical Analysis,
Faculty of Mathematics and Physics, 
Charles University,
So\-ko\-lo\-vsk\'a~83,
186~75 Praha~8,
Czech Republic}  
\email{musil@karlin.mff.cuni.cz}
\urladdr{0000-0001-6083-227X}

\date{\today}



\thanks{%
This research was partly supported
by the grant P201-13-14743S of the Grant Agency of the Czech Republic,
by the grant SVV-2016-260335 of the Charles University
and by the Neuron Fund for Support of Science.
}

\maketitle

\bibliographystyle{alpha}

\section{introduction and main results}

This paper is devoted to the study of strict $s$-numbers 
of certain type of Sobolev embedding.
As far as we know, all known results deal only with embeddings
under the assumption of compactness. Since some $s$-numbers
may be regarded as measurements of compactness, such
assumption seems to be reasonable. However not every
$s$-number has so strong connection with compactness.
For instance, Bernstein numbers represent a method of quantification
the finite strictly singularity,
the property weaker than compactness.

Let us take a look at known results. Denote by $Q$ a cube in
$\R^d$, let $1\le p < d$ and let $p\le q< \tfrac{dp}{d-p}$
or $p=d$ and $p\le q < \infty$. Then
\begin{equation*}
	b_n\bigl(V_0^1L^p(Q)\embed L^q(Q)\bigr)
		\asymp n^{-1/d}
\end{equation*}
and notice that the decay of $b_n$ is the same, regardless
how close to the limiting (i.e.\ non-compact) case we are
\citep[see][]{Ngu:15}.
Here ``$\asymp$'' means that left and right hand sides
are bounded by each other up to multiplicative constants
independent of $n$.
By $V_0^1X(Q)$, we mean the space of functions $u$ defined on $Q$
such that its continuation by zero outside $Q$ is
weakly differentiable and $|\nabla u|\in X(\Omega)$.
The norm is defined as $\|u\|_{V_0^1(Q)} = \|\nabla u\|_{X(Q)}$
and the relation ``$\embed$'' represents continuous inclusion.

This motivates us to ask the question what is
happening on this borderline. In our result, we show
that in the case when $q$ tends to infinity, the answer is in consent
with the forehead discussion.
In the situation when $q=\infty$, he proper domain space is not
$V_0^1L^d(Q)$, but the 
Sobolev-Lorentz space $V_0^1L^{d,1}(Q)$, the largest
rearrangement invariant Banach function space, which is
continuously embedded into $L^\infty(Q)$. Moreover,
every function in $V_0^1L^{d,1}(Q)$ has continuous representative,
hence
\begin{equation} \label{WC}
	V_0^1L^{d,1}(Q)\embed \C(Q).
\end{equation}
The embedding \eqref{WC} is thus an example of sharp non-compact
embedding which is the subject of our first main result.

\begin{theorem} \label{thm:ddim}
\def\SE{V_0^1L^{d,1}(Q)\embed \C(Q)}
Let $Q$ be a cube in $\R^d$, $d\ge 2$. Then for every $n\in\N$
\begin{equation} \label{mb4}
	s_n\bigl(\SE\bigr) \asymp n^{-\frac{1}{d}},
\end{equation}
where $s_n$ stands for $n$-th Bernstein 
or Isomorphism number
and
\begin{equation} \label{mb5}
	s_n\bigl(\SE\bigr) \asymp 1
\end{equation}
for $s_n$ representing Approximation, Gelfand or Kolmogorov numbers.
\end{theorem}

The definitions of the above-mentioned numbers are recalled in
Section~\ref{subsec:snumbers}.
The deep part of this result is contained in \eqref{mb4}. In particular,
the decay of Bernstein numbers to zero
implies that \eqref{WC} is finitely strictly singular.
The relation \eqref{mb5} is, on the other hand, rather clear
and we include it just for the sake of completeness.
As a consequence of our approach, we also get \eqref{mb4}
where $L^{d,1}$ is replaced by any smaller Banach function space.

\begin{corollary} \label{cor:WXC}
Let $X(Q)$ be any Banach function space
over the cube $Q$ in $\R^d$, $d\ge 2$,
satisfying $X(Q)\embed L^{d,1}(\Omega)$.
Then for every $n\in\N$
\begin{equation*}
	s_n\bigl(V_0^1X(Q)\embed \C(Q)\bigr) \asymp n^{-\frac{1}{d}},
\end{equation*}
in which $s_n$ stands for $n$-th Bernstein or Isomorphism number.
\end{corollary}

In \eqref{mb4} of Theorem~\ref{thm:ddim},
we do not focus on obtaining estimates on constants involved in ``$\asymp$'',
although they might be tracked from the proofs by careful reader.
The error is strongly dependent on the used method
and it is very unlikely to get sharp estimates.

Considering the one-dimensional case, the situation is completely
different. Using careful estimates, we obtain the exact values of
various $s$-numbers in the non-compact embedding $V_0^1L^1(I)\embed C(I)$.
The result reads as follows.

\begin{theorem} \label{thm:onedim}
\def\SE{V_0^{1,1}(I)\embed \C(I)}
Let $I$ be a bounded interval. Then for every $n\in\N$
\begin{equation} \label{mb1}
	s_n\bigl(\SE\bigr) = \frac{1}{2n},
\end{equation}
where $s_n$ stands for $n$-th Bernstein 
or Isomorphism number,
\begin{equation} \label{mb2}
	s_n\bigl(\SE\bigr) = \frac{1}{2},
\end{equation}
if $s_n$ is Approximation or Gelfand number
and for every $n\ge 2$
\begin{equation} \label{mb3}
	d_n\bigl(\SE\bigr) = \frac{1}{4},
\end{equation}
in which $d_n$ denotes the $n$-th Kolmogorov number.
\end{theorem}

The asymptotic behaviour of involved $s$-numbers are in the correspondence
with the above-mentioned vision.  Furthermore, we would like to notice the
difference between Gelfand and Kolmogorov numbers. This in particular
implies that the space $V_0^{1,1}(I)$, consisting of absolute continuous
functions with zero boundary condition, does not have the lifting property
\citep[cf.][p.~36]{Pin:85}.

Let us now briefly comment the method of proof.
The crucial ingredient of the proof of \eqref{mb1}
is the following proposition based on the deep result
\citep[Theorem~1]{Cha:09}.

\begin{proposition} \label{prop:CLL}
Let $E$ be a $n$-dimensional subspace of $\C(I)$ where $I$ is any
bounded interval. Then to every
$\varepsilon>0$
there exist a function $g\in E$, $\|g\|_\infty \le 1+\varepsilon$,
and $n$-tuple of points $t_1<t_2<\dots<t_n$ in $I$ such that
\begin{equation*}
	g(t_k) = (-1)^k
		\quad\text{for $1\le k\le n$}.
\end{equation*}
\end{proposition}

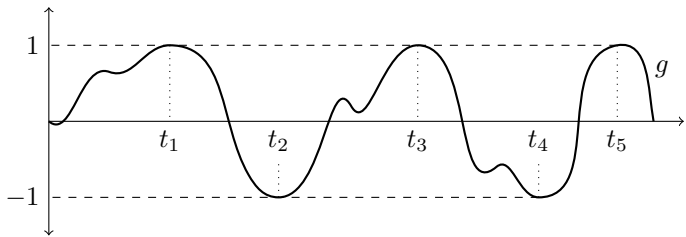
\begin{figure}[ht]
\begin{tikzpicture}[x=0.8cm,y=1cm]
\coordinate (A1) at (0,0);
\coordinate (A2) at (1,0.65);
\coordinate (A3) at (2,1);
\coordinate (A4) at (3.3,-0.7);
\coordinate (A5) at (3.8,-1);
\coordinate (A6) at (5,0.2);
\coordinate (A7) at (6.1,1);
\coordinate (A8) at (7.4,-0.6);
\coordinate (A9) at (8.1,-1);
\coordinate (A10) at (9.4,1);
\coordinate (A11) at (10,0);
\draw[thick]
	(A1) to[out=-40,in=180-20]
	(A2) to[out=-20,in=180]
	(A3) to[out=0,in=180-60]
	(A4) to[out=-60,in=180]
	(A5) to[out=0,in=180-60]
	(A6) to[out=-60,in=180]
	(A7) to[out=0,in=180+40]
	(A8) to[out=40,in=180]
	(A9) to[out=0,in=180+10]
	(A10) to[out=10,in=180-80] node[right]{$g$}
	(A11);
\draw[->] (A1)--(10.5,0);
\draw[<->] (0,-1.5) -- (0,1.5);
\draw[dotted] (A3) -- (A3|- 0,0) node[below] {$t_1$};
\draw[dotted, shorten >=5mm] (A5) -- (A5|- 0,0) node[below] {$t_2$};
\draw[dotted] (A7) -- (A7|- 0,0) node[below] {$t_3$};
\draw[dotted, shorten >=5mm] (A9) -- (A9|- 0,0) node[below] {$t_4$};
\draw[dotted] (A10) -- (A10|- 0,0) node[below] {$t_5$};
\draw[dashed] (A10) -- (A10-| 0,0) node[left] {$1$};
\draw[dashed] (A9) -- (A9-| 0,0) node[left] {$-1$};
\end{tikzpicture}
\caption{Every $n$-dimensional subspace of $\C(I)$ contains $n$ times
``zigzaging'' function.}
\end{figure}

The proof of \eqref{mb1} also relies on the linear
ordering of the domain set $I$.
In order to prove \eqref{mb4} we transfer this idea
into the higher dimension by using the approximation of
space-filling curve, namely the Hilbert
curve (see Figure~\ref{fig1}), for which we show that it
``preserves locality'' in some sense.

\pgfdeclarelindenmayersystem{Hilbert curve}{
  \rule{L -> +RF-LFL-FR+}
  \rule{R -> -LF+RFR+FL-}}

\begin{figure}[ht]
\begin{tikzpicture}[x=2.8cm,y=2.8cm]
		\newcommand*{\Order}{4}
		\newcommand*{\Step}{2.8cm}
		\foreach \i in {1,...,\Order}{
      \begin{scope}[xshift=1.2*\i*\Step,rotate=0,scale=1/(2^\i)]
			\draw[l-system={Hilbert curve, axiom=L, order=\i, step=\Step},
				shift={(0.5,0.5)}, thick]
				lindenmayer system;
			\draw[gray, thin, step=\Step ]
				(0,0) grid ({2^\i},{2^\i});
			\end{scope}}
\end{tikzpicture}
\caption{First four approximations of Hilbert curve in plane ($\mathcal H^2_1$ to $\mathcal H^2_4$).}
\label{fig1}
\end{figure}
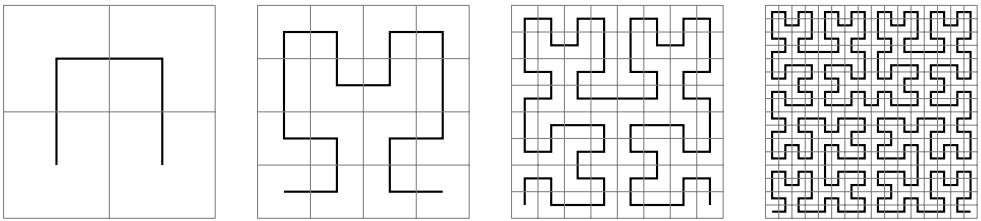

The paper is structured as follows. In Section~\ref{sec:back}
we recall the definitions we use and we collect all necessary later-needed
material. Section~\ref{sec:pre} is devoted to some preliminary
results including the proof of Proposition~\ref{prop:CLL},
Section~\ref{sec:onedim} deals with the one-dimensional
case contained in Theorem~\ref{thm:onedim}
and Section~\ref{sec:fulldim} treats the higher-dimensional
situation collected in Theorem~\ref{thm:ddim}.

\section{background material} \label{sec:back}

\subsection{Banach spaces}

For Banach spaces $X$ and $Y$, we denote by $B(X,Y)$ the set of all
bounded linear operators acting between $X$ and $Y$. For any $T\in B(X,Y)$, we
use just $\|T\|$ for its operator norm, since the domain and target spaces are
always clear from the context.  By $B_X$ and $S_X$, we mean the closed unit
ball of $X$ and unit sphere of $X$, respectively.

Let $Z$ be a closed subspace of the Banach space $X$. The quotient space $X/Z$
is the collection of the sets $[x] = x+Z = \{x+z;\,z\in Z\}$ equipped with the
norm
\begin{equation*}
	\| [x] \|_{X/Z} = \inf\{ \|x-z\|_X;\,z\in Z\}.
\end{equation*}
We sometimes adopt the notation $\|x\|_{X/Z}$ when no confusion is likely to
happen.  Recall the notion of canonical map $Q_Z\colon X\to X/Z$, given by
$Q_Z(x)=[x]$.

Let $Z$ be a closed subspace of a Banach space $X$.
Then
\begin{equation} \label{dual}
	Z^* \text{ is isometric to } X^*/Z^\perp,
\end{equation}
where
$Z^\perp = \{\phi\in X^*;\, \phi(z)=0 \text{ for all $z\in Z$}\}$
\citep[see][Proposition 2.6]{Fab:11}.

\subsection{The $s$-numbers} \label{subsec:snumbers}

Let $X$ and $Y$ be Banach spaces.
To a given operator $T \in B(X, Y)$, one can assign a scalar sequence $s_n(T)$,
$n\in\N$, satisfying for every $n\in\N$ the following conditions
\begin{enumerate}[(S1)]
	\item $\| T \| = s_1(T) \geq s_2 (T) \geq \dots \geq 0$,
 (monotonicity),
	\item $s_n(T + S) \leq s_n (T) + \|S\|$ for every $S \in B(X, Y)$,
	\item $s_n(B\circ T\circ A) \leq \| B \| s_n(T) \| A \|$
		for every $A \in B(X_1, X)$ and $B \in B(Y,Y_1)$,
 (ideal property),
	\item $s_n( \Id\colon \ell_n^2 \rightarrow \ell_n^2) = 1$,
 (norming property),
	\item $s_n(T) = 0$ whenever $\rank T< n$,
 (rank property).
\end{enumerate}
The number $s_n(T)$ is then called the $n$-th $s$-number of the operator $T$.
When (S4) is replaced by a stronger condition
\begin{enumerate}[(S1)] \setcounter{enumi}{5}
	\item $s_n( \Id\colon E\to E) = 1$ for every Banach space $E$, $\dim E = n$,
\end{enumerate}
we say that $s_n(T)$ is the $n$-th \textbf{strict} $s$-number of $T$. 
We also adopt the notation $s_n(T\colon X\to Y)$,
or $s_n(X\embed Y)$ when $T$ is identity, to emphasize the spaces
involved.

Note that the definition of $s$-numbers is not unified in literature.  The
original definition of $s$-numbers, which was introduced by Pietsch in
\citep{Pie:74}, uses the condition (S6) which was later modified to accommodate
wider class of $s$-numbers (like Weyl, Chang and Hilbert numbers).
Some authors also use more strict version of (S2).
For a detailed account of $s$-numbers, one is referred for instance to
\cite{Pie:07}, \cite{Car:90} or \citep{Lan:11}.

We shall recall some particular strict $s$-numbers.
Let $T \in B(X,Y)$ and $n \in\N$. The $n$-th Approximation, Gelfand,
Kolmogorov, Isomorphism 
and Bernstein numbers of $T$ are defined by
\begin{align*}
a_n(T) & = \inf_{\substack{ F \in B(X,Y) \\ \rank F < n}}
\| T-F \|,
	\\
c_n(T) & = \inf_{\substack{ M \subseteq X \\ \codim M < n}}
	\sup_{x\in B_M}  \| Tx  \|_Y,
	\\
d_n(T) & = \inf_{\substack{N \subseteq Y \\ \dim N < n}}
	\sup_{x\in B_X} \| Tx \|_{Y/N},
	\\
i_n(T) & = \sup  \| A \|^{-1} \| B \|^{-1},
\intertext{where the supremum is taken over all Banach spaces $E$ with $\dim E \geq n$ and
$A \in B(Y, E)$, $B \in B(E, X)$ such that $A\circ T\circ B$ is the identity
map on $E$, and}
	b_n(T) & = \sup_{\substack{M \subseteq X \\ \dim M \ge n}}
		\inf_{x \in S_M} \| Tx \|_Y,
\end{align*}
respectively.
It is known that Approximation numbers are the largest $s$-numbers
while the Isomorphism numbers are the smallest among all strict $s$-numbers.
Let us mention that for any $T\in B(X,Y)$ and any $n\in\N$ we have
\begin{equation*}
	i_n(T) \le 
		b_n(T)
		\le \max \{ c_n(T), d_n(T) \}
		\le a_n(T).
\end{equation*}
Moreover, if $X$ has lifting property, then $a_n(T)=d_n(T)$ for every $n\in\N$
\citep[see][]{Pie:07,Lan:11}.

\subsection{Function spaces}

Given an open set $\Omega\in\R^d$, we denote by $C(\Omega)$
the space of all continuous functions on $\bar \Omega$
equipped with the supremum norm.

We will, at some stage, use the setting of so-called Banach function spaces.
This wide
class of function spaces provides handy tools commonly used
in study of Sobolev embeddings. Roughly speaking, a Banach function
space $X(\Omega)$ over an measurable set $\Omega\subseteq \R^d$ is a collection of
all real-valued measurable functions, for which the functional
$\|\cdot\|_{X(\Omega)}$
is finite, where $\|\cdot\|_{X(\Omega)}$ is a Banach function norm
satisfying given set of axioms.
We also sometimes write $X$ instead of $X(\Omega)$ for brevity.
We kindly refer the reader who is not familiar with this notion
to the monograph \citep[Chapter 1 and 2]{Ben:88}.

Basic examples of Banach function spaces are Lorentz spaces
$L^{p,q}(\Omega)$,
where the norm is given by
\begin{equation*}
	\|f\|_{p,q} = 
		\biggl(
			p \int_{0}^{\infty} \bigl[\mu_f(s)\bigr]^{\frac{q}{p}} s^{q-1}\,\d s
		\biggr)^{\frac{1}{q}}.
\end{equation*}
Here, $\mu_f$ is the distribution function of $f$, defined as
\begin{equation*}
	\mu_f(t) = \bigl| \bigl\{ x\in\Omega;\, |f(x)| > t \bigr\}\bigr|
		\quad\text{for $0\le t<\infty$}.
\end{equation*}
We consider only the case $1\le q\le p$. When $p=q$, then $L^{p,p}=L^p$,
the customary Lebesgue space.
Recall that if $q\le p$ then
\begin{equation} \label{Lorsum}
	\sum_{k} \| f \chi_{E_k} \|_{p,q}^p \le \| f \|_{p,q}^p
\end{equation}
for any $f\in L^{p,q}(\Omega)$ and any pairwise disjoint measurable
subsets $E_k$ of $\Omega$.

Let $X(\Omega)$ be a Banach function space over $\Omega\subseteq \R^d$.
By Sobolev space $V_0^1 X(\Omega)$ we mean the collection of
all functions $u$ such that their continuation by zero outside $\Omega$
is weakly differentiable in $\R^d$ and satisfy $|\nabla u|\in X(\Omega)$.
On $V_0^1 X(\Omega)$ we use the norm
\begin{equation*}
	\|u\|_{V_0^1 X(\Omega)} = \|\nabla u\|_{X(\Omega)}.
\end{equation*}
When $X(\Omega)=L^p(\Omega)$, we use the abbreviate notation
$V_0^{1,p}(\Omega)$.
Recall that if $|\Omega|<\infty$ then, thanks to Poincar\'e
inequality, $\|u\|_{V_0^1X(\Omega)}$
is equivalent to the $\|u\|_{X(\Omega)}+\|\nabla u\|_{X(\Omega)}$,
the usual norm in Sobolev spaces denoted by $W_0^1X(\Omega)$.

\subsection{John domains and Sobolev embeddings}

We say that a bounded domain $\Omega\subseteq \R^d$ is a John domain
if there is a constant $C_J\ge 1$ and a distinguished point $x_0\in\Omega$ (called central point)
so that each point $x\in\Omega$ can be joined to $x_0$ by a curve (called John curve)
$\gamma\colon[0,1]\to \Omega$ such that $\gamma(0)=x$, $\gamma(1)=x_0$ and
\begin{equation} \label{defjohn}
	\dist\bigl(\gamma(t),\partial\Omega\bigr) \ge C_J^{-1} |x-\gamma(t)|
\end{equation}
for every $t\in[0,1]$.

The John domains provides a wide class of domains for which the
classical Poincar\'e inequality hold. We restrict our attention only to
the limiting case
\begin{equation*}
	\|u-u_\Omega\|_{\infty(\Omega)} \le C \|\nabla u\|_{d,1(\Omega)},
\end{equation*}
where we want to control the size of the constant $C$ appearing here.
Here, $u_\Omega$ represents the integral average of $u$ over $\Omega$.
It could be shown that $C$ depends on $C_J$, the John constant of
the domain $\Omega$, and its dimension $d$. However, this result seems
not to be available in the literature hence we sketch the proof here.

Let us come out of the potential estimate
obtained independently in \citep{Mar:88} and \citep{Res:80}
(see also \citep{Boj:88}) which reads as follows.
Suppose that $\Omega\subseteq \R^d$ is a John domain. Then
\begin{equation} \label{potentialest}
	|u(x) - u_\Omega|
		\le C(C_J,d) \int_{\Omega} \frac{|\nabla u(z)|}{|x-z|^{d-1}}\,\d z
\end{equation}
for every locally Lipschitz function $u$ in $\Omega$ and all $x\in\Omega$.
Assuming that $\nabla u\in L^{n,1}(\Omega)$, we can follow on the right hand
side of \eqref{potentialest} by
\begin{align*}
	\int_\Omega \frac{|\nabla u(z)|}{|x-z|^{d-1}}\,\d z
		& = \int_{0}^{\infty} \int_{\{|\nabla u|>s\}} |x-z|^{1-d} \,\d z\, \d s
			\cr
		& \le \int_{0}^{\infty}
				\bigl[\mu_{|\nabla u|}(s)\bigr]^{\frac{1}{d}}\, \d s
			= \frac{1}{d} \|\nabla u \|_{d,1(\Omega)}
\end{align*}
One can thus show that every weakly differentiable $u$ such that
$\nabla u\in L^{d,1}(\Omega)$ has
continuous representative which satisfies
\begin{equation} \label{sobinf1}
	\sup_{x,y\in\Omega} |u(x)-u(y)| \le C(C_J,d) \|\nabla u\|_{d,1(\Omega)}.
\end{equation}

For a different approach, where $C$ depends on the constants
in corresponding isoperimetric inequalities, we refer to \citep{Cia:98a}.

\section{Preliminaries} \label{sec:pre}

Throughout this section, assume that $I$ is any bounded interval.
In one-dimensional case, we can reduce the $s$-numbers of Sobolev embeddings
to the $s$-numbers of an integral operator. This is the objective of
the following proposition.

\def\DS{V_0^{1,1}}
\begin{proposition} \label{prop:WV}
Let $I$ be a bounded interval. Then for every $n\in\N$
\begin{equation*}
	s_n \bigl(\DS(I)\embed \C(I)\bigr)
		=	s_n \bigl( V\colon L^1_0(I)\to \C(I)\bigr),
\end{equation*}
where
\begin{equation*}
	Vf(t) = \int_{I\cap(-\infty,t)} f(s)\,\d s
		\quad\text{for $t\in I$}
\end{equation*}
and
$L^1_0(I)$ is a subspace of $L^1(I)$ given by
\begin{equation*}
	L^1_0(I) = \bigl\{ f\in L^1(I); \textstyle\int_I f =0 \bigr\}.
\end{equation*}
\end{proposition}

\begin{proof}
To prove ``$\le$'', consider the composition
\begin{equation*}
	L^1_0 \xrightarrow{V} \DS
		\embed \C.
\end{equation*}
By the ideal property (S3) of $s$-numbers,
\begin{equation*}
	s_n(V\colon L^1_0 \to \C)
		\le \|V\| s_n(\DS\embed \C)
\end{equation*}
and $\|V\|=1$ since
$\|Vf\|_{\DS} = \|f\|_{1}$
for every $f\in L^1_0$.

To prove the reverse inequality, we use the chain
\begin{equation*}
	\DS \xrightarrow{V^{-1}}	L^1_0
		\xrightarrow{V} \C.
\end{equation*}
Note that $V$ is one-to-one
mapping from $L^1_0$ onto $\DS$, $V^{-1}$ is well-defined
and thus
\begin{equation*}
	s_n(\DS\embed \C)
		\le s_n(V\colon L^1_0 \to \C) \|V^{-1}\|.
\end{equation*}
Next
$\|V^{-1}u\|_{L^1_0} = \|u'\|_{L^1_0} = \|u\|_{\DS}$,
whence $\|V^{-1}\|=1$.
\end{proof}

\begin{remark}
As a consequence of Proposition~\ref{prop:WV}, we have
the exact value of the norm of the embedding
$V_0^{1,1}(I) \embed \C(I)$, namely
\begin{equation*}
	\sup_{f\in V_0^{1,1}(I)} \frac{\|f\|_{\infty}}{\|f'\|_{1}}
		= \|V\| = \frac12.
\end{equation*}
\end{remark}

\begin{proof}[Proof of Proposition~\ref{prop:CLL}]
The unit ball $B_E$ is compact
and, by the Arzel\`a-Ascoli theorem,
equi-continuous. Thus, to a given $\varrho\in(0,1)$
there is a $\delta>0$ such that
\begin{equation} \label{Eequiint}
	\sup_{f\in B_E} |f(u)-f(v)| < \varrho
\end{equation}
for any $u,v \in[0,1]$ satisfying $|v-u|<\delta$.
To this $\delta$, choose points $s_k$ in $I$,
such that $s_0 < s_1 < \dots < s_{N+1}$
and $s_{k+1}-s_k < \delta$, $(0\le k\le N)$. We also require $N \ge n$.

Now, define a linear mapping $\Phi\colon E \to \ell^\infty_N$ by
\begin{equation*}
	(\Phi f )_k = f(s_k),
		\quad (1\le k\le N),
		\quad\text{for $f\in E$.}
\end{equation*}
We show that
\begin{equation} \label{Ephibound}
	(1-\varrho) \| f \|_\infty \le \|\Phi f \|_\infty \le \| f \|_\infty
	\quad\text{for $f\in E$.}
\end{equation}
The latter inequality is obvious.
As for the former, suppose that $f\in E$ is given
and attains its norm at some point $s\in I$. To such $s$,
there is a unique $k$ such that $s_k \le s < s_{k+1}$, whence
\begin{equation*}
	\|f\|_\infty - |f(s_k)|
		= |f(s)| - |f(s_k)|
		\le |f(s) - f(s_k)|
		\le \varrho \|f\|_\infty,
\end{equation*}
thanks to \eqref{Eequiint}. The inequality \eqref{Ephibound} therefore
follows by taking the supremum over $1\le k \le N$.
By \eqref{Ephibound}, $\Phi$ maps
$E$ onto $\Phi(E)$ isomorphically, whence $\Phi(E)$ forms a $n$-dimensional
subspace in $\ell^\infty_N$. By the Zigzag Theorem
\citep[Theorem 1]{Cha:09}, there is
an element $g\in E$ and indices $1\le k_1<k_2<\dots<k_n\le N$
such that $\|\Phi g\|_\infty = 1$ and
\begin{equation*}
	g(s_{k_j}) = (-1)^j
		\quad\text{for $1\le j\le n$}.
\end{equation*}
Also, by \eqref{Ephibound},
\begin{equation*}
	\|g\|_\infty \le \frac{1}{1-\varrho} \le 1+\varepsilon
\end{equation*}
for $\varrho$ taken sufficiently small.
The proposition now follows by taking $t_j = s_{k_j}$
for $1\le j\le n$.
\end{proof}

\begin{lemma} \label{lemm:lowKolm}
Let $T\colon X(I)\to \C(I)$ be any bounded linear operator acting on Banach
function space
$X(I)$ over an interval $I$
taking its values in continuous functions. Suppose that $f_k$, $(k\in\N)$,
is a collection in $X$ having the following properties.
\begin{enumerate}[\rm(i)]
\item It holds $\|f_k\|_X = 1$ for each $k\in\N$,
\item There is a point $\xi\in I$, two sequences $s_k\to \xi$ and $t_k\to\xi$
and two points $a<b$ in $I$ such that $Tf_k(s_k)\to a$ and $Tf_k(t_k)\to b$.
\end{enumerate}
Then we have the following lower bound for the Kolmogorov numbers for every $n\ge 2$.
\begin{equation} \label{eq:lowKolm}
	d_n(T\colon X\to \C) \ge \frac{b-a}{2}.
\end{equation}
\end{lemma}

\begin{proof}
Let $n\ge 2$ and any $\varepsilon>0$ be given.
By the definition of the $n$-th Kolmogorov number, there is a
subspace $N$ with $\dim N<n$ such that
\begin{equation*}
	d_n(T)+\varepsilon
		\ge \sup_{f\in B_{X}} \|Tf\|_{\C/N}.
\end{equation*}
Now, by the definition of the quotient norm, to every $f_k$ one can attach a
function $g_k \in N$
in a way that
\begin{equation*}
	\|Tf_k - g_k\|_\infty \le \|Tf_k\|_{\C/N} + \varepsilon.
\end{equation*}
Observe that the set of all the functions $g_k$ is bounded in $N$. Indeed,
\begin{align*}
	\|g_k\|_\infty
		\le \|Tf_k - g_k\|_\infty + \|Tf_k\|_\infty
		\le \|Tf_k\|_{\Linf/N} + \varepsilon + \|T\|\|f_k\|_X
		\le d_n(T) + 2\varepsilon + \|T\|
\end{align*}
for every $k\in\N$. Thus, since $N$ is finite-dimensional, there is a
convergent subsequence of $\{g_k\}$ which we denote $\{g_k\}$ again. Hence
$g_k$ converges to, say, $g\in N$, i.e., there is a index $k_0$ such
that $\|g_k - g\|_\infty < \varepsilon$ for every $k\ge k_0$.
The limiting function $g$ then satisfies
\begin{equation*}
	\|Tf_k - g\|_\infty \le \|Tf_k - g_k\|_\infty + \|g_k - g\|_\infty
		\le \|Tf_k\|_{\Linf/N} + 2\varepsilon
\end{equation*}
for $k\ge k_0$ and thus
\begin{equation} \label{eq:dnsup}
	\sup_{k\ge k_0} \|Tf_k - g \|_\infty
		\le d_n(V) + 3\varepsilon.
\end{equation}
Next, we estimate the left hand side of \eqref{eq:dnsup} by taking the value
attained at the points $s_k$, i.e.,
\begin{equation} \label{eq:eval1}
	\sup_{k\ge k_0} \|Tf_k - g \|_\infty
	\ge \sup_{k\ge k_0} |Tf_k(s_k) - g(s_k) |
	\ge |a-g(\xi)|,
\end{equation}
or at the points $t_k$, i.e.,
\begin{equation} \label{eq:eval2}
	\sup_{k\ge k_0} \|Tf_k - g \|_\infty
	\ge \sup_{k\ge k_0} |Tf_k(t_k) - g(t_k) |
	\ge \bigl|b-g(\xi)\bigr|,
\end{equation}
where we used the continuity of $g$.
Combining \eqref{eq:dnsup}, \eqref{eq:eval1} and \eqref{eq:eval2}, we get
\begin{equation*}
	d_n(V) + 3\varepsilon
	\ge \max \bigl\{ |a-g(t)|, |b-g(t)| \bigr\}
	\ge \tfrac{b-a}{2}.
\end{equation*}
Therefore we obtain \eqref{eq:lowKolm} by the arbitrariness of $\varepsilon$.
\end{proof}

\section{Embeddings on the interval} \label{sec:onedim}

Throughout this section, all function spaces will be defined over
the closed unit interval unless explicitly stated.

\begin{lemma} \label{lemm:Viso}
Let $n\in\N$. Then
\begin{equation*}
	i_n(V\colon \Lio\to\C) \ge \frac{1}{2n}.
\end{equation*}
\end{lemma}

\begin{proof}
Let $n\in\N$ be fixed. Consider the chain
\begin{equation*}
	\linfn
		\xrightarrow {B} \Lio
		\xrightarrow {V} \C
		\xrightarrow {A} \linfn
\end{equation*}
where we define $A$ as
\begin{equation*}
	(Af)_k = f\biggl(\frac{2k-1}{2n}\biggr),
		\quad (1\le k\le n),
		\quad\text{for $f\in\C$.}
\end{equation*}
and $B$ by
\begin{equation*}
	B(\vx) = 2n \sum_{k=1}^n x_k \bigl( \chi_{I_{2k-1}} - \chi_{I_{2k}} \bigr)
		\quad\text{for $\vx\in\lin$},
\end{equation*}
where the intervals $I_1$, $I_2$, \dots, $I_{2n}$ are the non-overlapping intervals
of the same length~$\frac{1}{2n}$, i.e.,
\begin{equation*}
	I_k = \biggl[ \frac{k-1}{2n}, \frac{k}{2n} \biggr]
		\quad\text{ for $1\le k \le 2n$.}
\end{equation*}
Both $A$ and $B$ are well-defined and the composition $A\circ V\circ B$ is
an identity mapping on $n$-dimensional space $\linfn$. The operator norms
are $\|A\| = 1$ and $\|B\| = 2n$ thus, by the definition
of the $n$-th Isomorphism number,
\begin{equation*}
	i_n(V)\ge \|A\|^{-1} \|B\|^{-1} \ge \frac{1}{2n}.
\end{equation*}
\end{proof}

\begin{lemma} \label{lemm:VBern}
Let $n\in\N$. Then
\begin{equation*}
	b_n(V\colon \Lio\to\C) \le \frac{1}{2n}.
\end{equation*}
\end{lemma}

\begin{proof}
By the definition of the $n$-th Bernstein number,
it is enough to show that for any
$\varrho>0$ and any $n$-dimensional subspace $E$ of $\Lio$ satisfying
\begin{equation} \label{lbound}
	\|Vf\|_\infty \ge \varrho \|f\|_1
		\quad\text{for every $f\in E$}
\end{equation}
we have
\begin{equation*}
	\varrho \le \frac{1}{2n}.
\end{equation*}
To prove this, fix some arbitrary $\varepsilon\in(0,1)$.
Since $V$ is linear and injective, $V(E)$ is a $n-$dimensional subspace
of $\C$. By Proposition~\ref{prop:CLL}, there is an element
$h\in \Lio$ such that $\|Vh\|_\infty \le 1+\varepsilon$
and
\begin{equation*}
	Vh(t_k) = (-1)^k
		\quad\text{for $1\le k\le n$}.
\end{equation*}
for some
points $0\le t_1<t_2<\dots<t_n\le 1$.
Let us estimate the norm $\|h\|_1$. We have
\begin{align*}
	\int_{0}^{1} |h(s)|\,\d s
		& = \int_{0}^{t_{1}} |h(s)|\,\d s
			+ \sum_{k=1}^{n-1} \int_{t_{k}}^{t_{k+1}} |h(s)|\,\d s
			+ \int_{t_{n}}^{1} |h(s)|\,\d s
		\\
		& \ge |Vh(t_{1})|
			+ \sum_{k=1}^{n-1} |Vh(t_{k+1}) - Vh(t_{k})|
			+ |Vh(t_{n})|
		\\
		& = 1+2(n-1)+1=2n.
\end{align*}
Note, that for the third term, we used that $\int_{t_n}^1 h = \int_0^{t_n} h$
thanks to the boundary condition in $\Lio$.
Next, by \eqref{lbound},
\begin{align*}
	1+\varepsilon
		\ge \|Vh\|_\infty
		\ge \varrho \int_{0}^{1} |h(s)|\,\d s,
\end{align*}
whence
\begin{equation*}
	\varrho \le \frac{1+\varepsilon}{2n}
\end{equation*}
and the result follows by the arbitrariness of $\varepsilon$.
\end{proof}

\begin{lemma} \label{lemm:VGelf}
Let $n\in\N$. Then
\begin{equation*}
	c_n(V\colon\Lio\to\C) \ge \frac{1}{2}.
\end{equation*}
\end{lemma}

\begin{proof}
We will use an alternative expression for Gelfand numbers,
proved in \citep[Proposition~2.3.2]{Car:90}, which proposes
that $c_n(V)$ is the infimum of all $\varrho>0$ such that
\begin{equation} \label{carldef}
	\|Vf\|_\infty \le \sup_{1\le k\le m} |\varphi_k(f)| + \varrho \|f\|_1
		\quad\text{for all $f\in\Lio$},
\end{equation}
where $m<n$ and $\varphi_k$, ($1\le k\le m$), are any continuous linear functionals on $\Lio$.

First, we explicitly describe all possible $\varphi$, i.e., we find the
dual space of $\Lio$. Since $\Lio$ is closed subspace of $\Li$, by
\eqref{dual}, its dual is isometric to
the quotient
\begin{equation*}
	(\Lio)^* \simeq \Lnek/(\Lio)^{\perp}.
\end{equation*}
Here, we have
\begin{align*}
(\Lio)^\perp
	& \simeq \bigl\{ g\in \Lnek;\,
		\textstyle\int_0^1 fg=0 \text{ for all }f\in\Lio \bigr\}
		\\
	& \simeq \{\text{$g$ is a constant function}\}
		\simeq \R,
\end{align*}
so $(\Lio)^*$ is isometric to $\Lnek/\R$ via the
mapping $\sigma\colon \Lnek/\R \to \Lio^*$ given by
\begin{equation*}
	\sigma_{[g]}(f) = \int_{0}^{1} f(t)g(t)\,\d t
		\quad\text{ for $[g]\in\Lnek/\R$ and $f\in\Lio$.}
\end{equation*}
Note that $\sigma$ is well-defined, i.e., the definition does not depend on the
choice of the representative $g$, since,
if $[g_1]=[g_2]$, then there is a constant $c$ such that
$g_1 = g_2 + c$ a.e.\ and hence
\begin{equation*}
	\sigma_{[g_1]}(f)
		= \int_{0}^{1} f(t)g_1(t)\,\d t
		= c \int_{0}^{1} f(t)\,\d t
		+ \int_{0}^{1} f(t)g_2(t)\,\d t
		= \sigma_{[g_2]}(f)
\end{equation*}
for every $f\in\Lio$ and therefore $\sigma_{[g_1]}=\sigma_{[g_2]}$.
In conclusion, \eqref{carldef} rewrites as
\begin{equation} \label{carldef2}
	\|Vf\|_\infty
		\le \sup_{1\le k\le m}\,
		\biggl| \int_{0}^{1} f(t)g_k(t)\,\d t \biggr|
		+ \varrho \|f\|_1,
\end{equation}
where $m<n$ and $g_k$ are bounded functions.

Next, fix $\varrho>0$, $m<n$ and any $g_k\in\Lnek$, ($1\le k\le m$),
such that \eqref{carldef2} holds for every $f\in\Lio$. We
have to show that $\varrho\ge \tfrac{1}{2}$.
Let $\varepsilon>0$
and denote $M=\max_{1\le k\le m} \|g_k\|_\infty$.
All the values of the functions $|g_k|$ thus essentially range in
the interval $[0,M]$.
Let $I_1, I_2,\dots,I_r$ be sets of diameter less than $\varepsilon$ 
and covering $[0,M]$.
Denote by $M_{k,l}$ the preimage of $I_l$ under $|g_k|$, i.e.,
\begin{equation*}
	M_{k,l} = |g_k|^{-1}[I_l],
	\quad (1\le k \le n, 1\le l\le d).
\end{equation*}
Clearly
\begin{equation} \label{cupcap}
	\bigcup_{\pi\in\Pi_d} \bigcap_{k=1}^n M_{k,\pi(k)}
		= \bigcap_{k=1}^n \bigcup_{l=1}^r M_{k,l}
		= [0,1]
		\quad\text{a.e.,}
\end{equation}
where $\Pi_r$ denotes the set of all the permutations of the elements
$1,2,\dots,r$.
Therefore, by \eqref{cupcap}, there is a $\pi\in\Pi_r$, such that
\begin{equation*}
	\Omega = \bigcap_{k=1}^n M_{k,\pi(k)}
\end{equation*}
has positive measure.
Now, fix any $i_l \in I_l$, $(1\le l\le r)$. We have
\begin{equation} \label{aconst}
	\bigl| |g_k| - i_{\pi(k)} \bigr| \le \diam I_{\pi(k)} < \varepsilon
		\quad\text{on $\Omega$},
		\quad (1\le k\le n).
\end{equation}
One can also choose a point $x$ in $[0,1]$ such that both
$\Omega_1 = (0,x)\cap \Omega$ and $\Omega_2 = (x,1)\cap \Omega$
have positive measure. Now, there is a function $f$ having support
in $\Omega$, positive on $\Omega_1$, negative on $\Omega_2$
and satisfying
\begin{equation} \label{om1om2}
		\int_{\Omega_1} f(t)\,\d t
		= - \int_{\Omega_2} f(t)\,\d t
		= \frac{1}{2}.
\end{equation}
Observe that also $\|f\|_1=1$ and 
\begin{equation} \label{half}
	Vf(x) = \int_{0}^{x} f(t)\,\d t = \frac{1}{2}.
\end{equation}
Since, by \eqref{om1om2}, $\int_0^1 f = 0$ and, thanks to \eqref{aconst},
$g_k$ is almost constant on $\Omega$, we claim that $\int_0^1 fg_k$ is small.
Indeed, for any $1\le k\le n$,
\begin{align}
\begin{split} \label{small}
	\int_{0}^{1} f(t)g_k(t)\,\d t
		& = \int_{\Omega_1} f(t)g_k(t)\,\d t
			+ \int_{\Omega_2} f(t)g_k(t)\,\d t
			\\
		& \le (i_{\pi(k)} + \varepsilon) \int_{\Omega_1} f(t)\,\d t
			+ (i_{\pi(k)} - \varepsilon) \int_{\Omega_2} f(t)\,\d t
			\\
		& \le \frac12(i_{\pi(k)} + \varepsilon)
			-\frac12(i_{\pi(k)} - \varepsilon)
		\le \varepsilon.
\end{split}
\end{align}
Therefore, by \eqref{carldef2} together with \eqref{half} and \eqref{small},
\begin{equation*}
	\frac{1}{2} = Vf(x)
		\le \|Vf\|_\infty
		\le \sup_{1\le k\le m}\,
		\biggl| \int_{0}^{1} f(t)g_k(t)\,\d t \biggr|
		+ \varrho \|f\|_1
		\le \varepsilon + \varrho
\end{equation*}
and $\varrho\ge\tfrac{1}{2}$ by the arbitrary choice of $\varepsilon$.
\end{proof}

\begin{lemma} \label{lemm:VKolm}
Let $n\ge 2$. Then
\begin{equation*}
	d_n(V\colon \Lio\to\C) = \frac{1}{4}.
\end{equation*}
\end{lemma}

\begin{proof}
Let us first establish the inequality ``$\le$''.
By the monotonicity of $s$-numbers, it is enough to show that $d_2(V)\le\frac{1}{4}$. 
Let $N$ be the one-dimensional subspace of $\Linf$
consisting of constant functions.
Then, by the definition of the second Kolmogorov number,
\begin{equation*}
	d_2(V)\le \sup_{f\in B_\Lio} \inf_{c\in \R} \|Vf-c\|_\infty.
\end{equation*}
Now, to a given $\varepsilon>0$, there is $f\in B_\Lio$ such that
\begin{equation*}
	d_2(V) - \varepsilon
		\le \inf_{c\in\R} \|Vf-c\|_\infty.
\end{equation*}
We claim that
\begin{equation*}
	|Vf(u)-Vf(v)|\le\frac{1}{2}
\end{equation*}
for any $u,v\in [0,1]$. To observe that, let $E$ be any measurable subset of $[0,1]$
and $E^c$ its complement. Then
\begin{align*}
	\int_E f 
		& = \int_E f(t)\,\d t  - \frac{1}{2} \int_{E\cup E^c} f(t)\,\d t
		\\
		& = \frac{1}{2} \int_E f(t)\,\d t - \frac{1}{2} \int_{E^c} f(t)\,\d t
		\\
		& \le \frac{1}{2} \int_{0}^{1} |f(t)|\,\d t
		\le \frac{1}{2}
\end{align*}
and the claim follows once we set $E$ as the interval with endpoints $u$ and $v$.
Next, on setting
\begin{equation*}
	c = \frac{1}{2}(\sup_{[0,1]} Vf - \min Vf),
\end{equation*}
we obtain that
\begin{equation*}
	\|Vf - c\|_\infty \le \frac{1}{4}.
\end{equation*}
Since $\varepsilon$ was arbitrary, we have $d_n(V)\le \frac{1}{4}$.

For the opposite inequality, 
consider the functions
\begin{equation*}
	f_k = 2^{k}\bigl( 
		\chi_{(2^{-k-1},2^{-k})}
		- \chi_{(1-2^{-k},1-2^{-k-1})}
		\bigr)
	\quad\text{for $k\in\N$}.
\end{equation*}
These are linearly independent elements of $\Lio$ with $\|f_k\|_1=1$ and
satisfying $Vf_k(0)=0$ and $Vf_k(2^{-k})=\frac{1}{2}$.
The assertion thus follows by Lemma~\ref{lemm:lowKolm}.
\end{proof}

\begin{proof}[Proof of Theorem~\ref{thm:onedim}]
By Proposition~\ref{prop:WV}, the $s$-numbers of
Sobolev embedding $V_0^{1,1}(I)\embed \C(I)$
coincide with $s$-numbers of the operator
$V\colon L_0^1(I) \to \C(I)$.
We can assume that $I$ is the unit interval
because in the non-compact setting the scaling
does not affect the norm and the values of $s$-numbers.

The upper bound for the Bernstein numbers $b_n(V)\le\tfrac{1}{2n}$ follows from
Lemma~\ref{lemm:VBern} and the lower
bound for the Isomorphism numbers reads as $i_n(V)\ge \tfrac{1}{2n}$, thanks to
Lemma~\ref{lemm:Viso}.  Since the Isomorphism numbers are the smallest among
strict $s$-numbers, the equalities in \eqref{mb1} follow.

Next, since the Approximation numbers are the largest $s$-numbers, we have
$\tfrac12=\|V\|\ge a_n(V) \ge c_n(V) \ge \tfrac12$, where the last inequality
is due to Lemma~\ref{lemm:VGelf}. This proves \eqref{mb2}. Finally, the
relation \eqref{mb3} is subject of Lemma~\ref{lemm:VKolm}.
\end{proof}

\section{Higher-dimensional Sobolev embeddings} \label{sec:fulldim}

In this section we denote by $Q=(0,1)^d$ 
the $d$-dimensional cube where $d\ge 2$.
All the function spaces are considered
over $Q$, unless explicitly stated.

\def\DS{V_0^1L^{d,1}}

\begin{lemma} \label{lemm:Wiso}
\def\DS{V_0^1X}
Let $X$ be any Banach function space satisfying
$X \subseteq \Ldi$.
Then there exists a positive constant $c$, such that
for every $n\in\N$
\begin{equation*}
	i_n (\DS\embed C) \ge cn^{-\frac{1}{d}}.	
\end{equation*}
\end{lemma}

\begin{proof}
\def\DS{V_0^1X}
Let $n=m^d$ for some $m\in\N$. Denote $B_1$, $B_2$, \dots, $B_n$
disjoint balls in $Q$ having radii $r=\tfrac{1}{2m}$ and.
Let us label the centre point of $B_k$ by $x_k$, ($1\le k\le n$).
On setting
\begin{equation*}
	u_k(x) = \bigl( r - |x-x_k| \bigr) \chi_{B_k}
		\quad\text{for $x\in Q$,}
\end{equation*}
we have that $\|u_k\|_\infty = 1$
and $|\nabla u_k|=\chi_{B_k}$ a.e.\ for every $1\le k\le n$.
Consider now the chain
\begin{equation*}
	\linfn \xrightarrow{B} \DS
		\embed \C
		\xrightarrow{A} \linfn,
\end{equation*}
where $A$ is defined as
\begin{equation*}
	(Au)_k = u(x_k),
		\quad(1\le k\le n),
		\quad\text{for $u\in \C$}
\end{equation*}
and $B$ is given by
\begin{equation*}
	B(\vy) = \frac{1}{r} \sum_{k=1}^n y_ku_k
		\quad\text{for $\vy\in\linfn$}.
\end{equation*}
Both $A$ and $B$ are well-defined and the composition $A\circ \Id\circ B$ forms
an identity on $\linfn$. As for the norms, clearly $\|A\|=1$ and
\begin{equation*}
	\|B(\vy)\|_{\DS}
		= \frac{1}{r} \biggl\| \sum_{k=1}^n y_k \nabla u_k \biggr\|_{X}
		\le \frac{1}{r} \|\vy\| \biggl\| \sum_{k=1}^n \nabla u_k \biggr\|_{X}
		\le \frac{\|\chi_Q\|_X}{r}\|\vy\|,
\end{equation*}
therefore, by the definition of the $n$-th Isomorphism number,
\begin{equation*}
	i_n(\DS\embed \C)
		\ge \frac{1}{\|A\| \|B\|}
		\ge \frac{r}{\|\chi_Q\|_X}
		= \frac{1}{2\|\chi_Q\|_X}n^{-\frac{1}{d}}.
\end{equation*}
The result for arbitrary $n$ then follows by the monotonicity (S1) of
the Isomorphism numbers.
\end{proof}

\begin{lemma}\label{lemm:WBern}
There is a positive constant $c=c(d)$ such that for every $n\in\N$
we have the following estimate of the Bernstein numbers
\begin{equation*}
	b_n(\DS\to\C) \le c n^{-\frac{1}{d}}.
\end{equation*}
\end{lemma}

Before we prove this lemma, we will make a preliminary
discussion.
Let $d$ and $k$ be given. We shall work with the dyadic cubes in $Q$.
For $j=0,1,\dots,k$ we denote by $\mathcal Q_j$ the family of all the dyadic cubes
in $Q$ with side length $2^{-j}$, i.e.,
\begin{equation} \label{dicubes}
	\mathcal Q_j = \bigl\{ \{x\in Q;\,2^jx- z \in Q\};\, z\in\Z^d \bigr\}.
\end{equation}
We shall also refer to $\mathcal Q_j$ as the $j$-th generation of dyadic cubes in $Q$.

We would like to define an arrangement of the cubes in $\mathcal Q_k$
in a way that we could transform the cube $Q$ into the straight strip.
We moreover require that every subsequence of consecutive cubes
forms a set which is ``plump'' enough.
More precisely, denote by $Q\colon \{1,2,\dots,2^{dk}\}\to \mathcal Q_k$
a numbering of cubes in $\mathcal Q_k$ with the property
that consecutive cubes $Q_j$ and $Q_{j+1}$ neighbourhood by a face
and define
\def\Oij{\Omega_{ij}}
\begin{equation} \label{Oij}
	\Oij = \bigcup_{l=i}^j Q_l
		\quad\text{for $1\le i\le j\le 2^{dk}$}.
\end{equation}
We would like to have the inequality
\begin{equation} \label{sobinf3}
	\sup_{x,y\in\Oij} |u(x)-u(y)| \le C \|\nabla u\|_{d,1(\Oij)}
\end{equation}
valid for every continuous $u$ having $|\nabla u|\in L^{d,1}(\Omega)$
and for every $1\le i\le j\le 2^{dk}$
with the absolute constant $C$ independent of $i$, $j$ and $k$.

By \eqref{sobinf1}
the constant $C$ in \eqref{sobinf3} might be controlled by the
John constant $C_J$ of the John domain \eqref{defjohn}
and by the dimension $d$.
Therefore it is enough to have $C_J(\Oij)$ uniformly bounded in
$i$, $j$ and $k$.
However not every arrangement $Q$ may possess this property. For instance,
one has to avoid the numbering ``by lines'' since any such thin joist
would have the constant $C_J$ proportional to $2^k$.

As an example of proper numbering we introduce the arrangement
inspired by the $k$-th approximation of Hilbert space-filling curve 
which we denote by
\def\Hdk{\mathcal H^d_k}
$\Hdk\colon [1,2^{dk}] \to Q$.
Such mappings are well-studied usually in the field of computer science
where they are used for instance in an implementation of a multiattribute
access methods, data clustering \citep{Moo:01} or as a heuristics for Traveling Salesman Problem \citep{Pla:89}.

Unfortunately, it is not easy to give the exact definition of $\Hdk$ and
we will not do it here. For those readers who are interested
in this topic, we refer to \citep{But:69} or \citep{Alb:00}
for analytical definition
or \citep{Moo:01} for recurrent geometric view.
We will use only one basic fact of those curves which can be formulated as follows.

\begin{quotation}
Once the curve $\Hdk$ enters any cube $\widetilde Q$ from the family $\mathcal
Q_l$, it does not leave $\widetilde Q$ until it visits all the subcubes of
$\widetilde Q$ contained in the refined family $\mathcal Q_{l+1}$
(and hence $\mathcal Q_m$ for any $m>l$). For illustration, see Figure~\ref{fig1}.
\end{quotation}

Let us summarize the above-mentioned thoughts into the following lemma.

\begin{lemma} \label{lemm:John}
Let $Q=(0,1)^d$, $d\ge 2$, and let $\mathcal Q_k$ be the family of dyadic
cubes \eqref{dicubes} for some $k\in\N$.
Suppose that the cubes in $\mathcal Q_k$ are labeled along the
path $\Hdk$, i.e., set $Q\colon\{1,2,\dots,2^{dk}\} \to \mathcal Q_k$
as $Q_l = \widetilde{Q}$ where $\widetilde{Q}\in\mathcal Q_k$ is such that $\Hdk(l)\in \widetilde{Q}$.
Then all the sets $\Oij$ defined by \eqref{Oij} are John domains
with comparable constants depending on $d$ and independent of $k$, $i$ and $j$,
i.e., there is $C=C(d)$ such that
\begin{equation*}
	C_J(\Oij) \le C
		\quad\text{for every $1\le i\le j\le 2^{dk}$.}
\end{equation*}
\end{lemma}

\begin{proof}
Let $d$, $k$, $i$ and $j$ be fixed.  We may imagine $\Oij$ as a chain of
``small cubes'' from the youngest family $\mathcal Q_k$. As the chain
twists, it also fills the larger dyadic cubes from older families
$\mathcal Q_m$, ($m<k$).

Denote by $\mathcal Q_{k_1}$ the oldest family for which there exists any
dyadic cube contained in $\Oij$ and label any of such cube as $Q_1^{k_1}$.  The
curve $\Hdk$ passes through $Q_1^{k_1}$ and exits by two different faces. Let
us follow $\Hdk$ on one of its exists (the other one would be analogous) and
watch for the largest filled dyadic cubes. We may still meet the dyadic cubes
from $\mathcal Q_{k_1}$. If so, denote $m_1$ its total amount and label them by
$Q_2^{k_1}$, \dots, $Q_{m_1}^{k_1}$. Observe that $m_1\le 2(2^d-2)$.  Continuing
our journey along $\Hdk$, the next largest filled dyadic cube must belong
to younger family $\mathcal Q_{k_2}$, ($k_2>k_1$). Denote by $Q_{1}^{k_2}$, \ldots,
$Q_{m_2}^{k_2}$ the cubes in $\mathcal Q_{k_2}$ the curve $\Hdk$ fills
consecutively. Their total amount satisfies $m_2\le 2^d-1$. Let us continue
in this manner inductively. We obtain the sequence of families $\mathcal Q_{k_1}$, \dots, 
$\mathcal Q_{k_r}$ and the sequences of dyadic cubes
$\{ Q_1^{k_1},\dots,Q_{m_1}^{k_1}\}\subseteq \mathcal Q_{k_1}$,
\dots, $\{Q_1^{k_r},\dots Q_{m_r}^{k_r}\}\subseteq \mathcal Q_{k_r}$.
Except for the $k_1$-th family, all such sequences are no longer than $2^d-1$,
i.e., $m_l\le 2^d-1$, ($2\le l\le r$). See Figure~\ref{fig2} for illustration.
\begin{figure}[ht]
\begin{tikzpicture}[x=2.3cm,y=2.3cm]
	\newcommand*{\Size}{1}
	\newcommand*{\Shift}{0.25}
	\def\zn{\tiny$\bullet$}

	\coordinate (S) at (\Shift,0);
	\coordinate (R) at (\Size,0);
	\coordinate (T) at (0,\Size);
	\coordinate (U) at ($ (R) + (T) $);
	\foreach \y in {R,T,U}
		\foreach \x in {2,4,8,16}
			\coordinate (\y\x) at ($ 1/\x*(\y) $);
	\coordinate (L1) at (0,0);
	\coordinate (P1) at ($ (R) + (S) $);
	\coordinate (L2) at ($ (P1) + (R) $);
	\coordinate (P2) at ($ (L2) + (R2) + (S) $);
	\coordinate (L3) at ($ (P2) + (R2)$);
	\coordinate (P3) at ($ (L3) + (R4) + (S) $);
	\coordinate (L4) at ($ (P3) + (R4) + 2*(S) + (R8)$);
	\coordinate (P4) at ($ (L4) + (R8) + (S) $);

	\foreach \y in {L,P} {
		\coordinate (u) at (U);
		\foreach \x in {1,...,4} {
			\draw (\y\x) rectangle ++(u);
			\coordinate (u) at ($ 1/2*(u) $);
			\coordinate (I\y\x) at ($ (\y\x) + (u) $);
			\node at (I\y\x){\zn};
		}
	}
	\draw (IL1) -- ($(T2)+(R)$);
	\draw ($(P1)+(T2)$) -- (IP1)
		-- ($(L2)+(T4)$) -- (IL2)
		-- ($(L2)+(R2)+(T4)$);
	\draw ($(P2)+(T4)$) -- (IP2)
		-- ($(L3)+(T8)$) -- (IL3)
		-- ($(L3)+(R4)+(T8)$);
	\draw ($(P3)+(T8)$) -- (IP3) -- ($(P3)+(R4)+(T16)$);
	\draw ($(L4)+(T16)$) -- (IL4) -- ($(L4)+(T16)+(R8)$);
	\draw ($(P4)+(T16)$) -- (IP4);

	\foreach \x in {1,...,4}
		\node at ($1/2*(IL\x)+1/2*(IP\x)$) {\small\dots};
	\node at ($(P3)+(R4)+1/2*(S)+(T16)$) {\small\dots};
	\node at ($(L4)-1/2*(S)+(T16)$) {\small\dots};

	\coordinate (r) at (R);
	\foreach \x in {1,...,3}{
		\coordinate (r) at ($1/2*(r)$);
		\node[below] at ($(L\x)+(r)$) {$Q_1^{k_\x}$};
		\node[below] at ($(P\x)+(r)$) {$Q_{m_\x}^{k_\x}$};
	}
	\node[below] at ($(L4)+(R16)$) {$Q_1^{k_r}$};
	\node[below] at ($(P4)+(R16)$) {$Q_{m_r}^{k_r}$};

	\node[above right, xshift=-1.5mm, yshift=-0.5mm] at (IL1) {$x_1^{k_1}$};
	\node[above right, xshift=-1.5mm, yshift=-0.5mm] at (IP1) {$x_{m_1}^{k_1}$};
	\node[above right, xshift=-1.5mm, yshift=-0.5mm] at (IL2) {$x_1^{k_2}$};
	\node[above right, xshift=-1.5mm, yshift=-0.5mm] at (IP2) {$x_{m_2}^{k_2}$};
\end{tikzpicture}
\caption{Labeling of dyadic cubes in $\Oij$}
\label{fig2}
\end{figure}
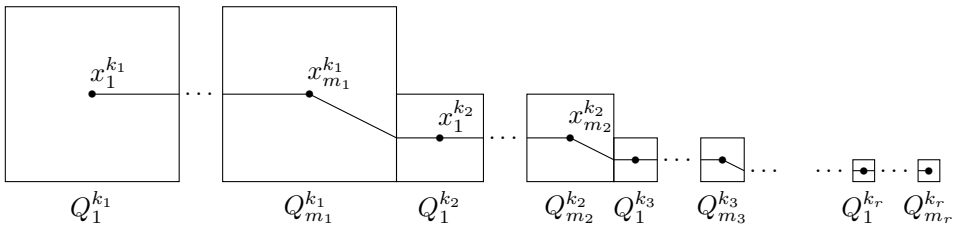

Next, we construct a John curve for a given $x\in\Oij$. We may assume that
$x\in Q_{m_r}^{k_r}$ since otherwise we can stop the process of the labeling
cubes as we reach the first dyadic cube containing $x$.
Denote by $x_1^{k_l},\dots, x_{m_l}^{k_l}$ the central points of the cubes
$Q_1^{k_l}, \dots,Q_{m_l}^{k_l}$, ($1\le l\le r$), respectively.
The image of the curve $\gamma\colon[0,1]\to\Oij$ is now defined
as a polyline as follows.
We choose $x$, the starting point, and we connect it to
the center $x_{m_r}^{k_r}$, then we join all the points
$x_\xi^{k_r}$, ($1\le\xi\le m_r$), in the reverse order.
From $x_1^{k_r}$ we go to the point $x_{m_{r-1}}^{k_{r-1}}$ through
the center of the face of $Q_{1}^{k_r}$ which
neighbourhoods with $Q_{m_{r-1}}^{k_{r-1}}$.
Denote this point by $y^{k_r}$ and similarly all the others by
$y^{k_2},\dots,y^{k_{r-1}}$.
We then follow this manner until we finally reach the point $x_1^{k_1}$,
the central point of the John domain $\Oij$. See Figure~\ref{fig2}
again.

Let now $\gamma$ be parametrized arbitrarily and choose any $t\in[0,1]$.
We distinguish two cases. If $\gamma(t)$ lays between $x$ and $x_{m_r}^{k_r}$
then \eqref{defjohn} holds with $C_J=C_J(Q)$ trivially.
In the remaining cases, assume that $\gamma(t)\in Q_{m}^{k_l}$
for some $1\le l\le r$ and some $1\le m\le m_l$.
We have, by the construction,
\begin{align*}
	|x-\gamma(t)|
		&\le
			\sum_{\lambda=l+1}^{r} \sum_{\xi=1}^{m_\lambda-1}
				|x_\xi^{k_\lambda} - x_{\xi+1}^{k_\lambda}|
			+ \sum_{\xi=m}^{m_l-1} | x_{\xi}^{k_l} - x_{\xi+1}^{k_l} |
				\cr
		&\quad 
			+ \sum_{\lambda=l+1}^{r} |x_1^{k_{\lambda}} - y^{k_\lambda} |
			+ \sum_{\lambda=l+1}^{r}
				|y^{k_\lambda} - x_{m_{\lambda-1}}^{k_{\lambda-1}} |
			+ | x_m^{k_l} - \gamma(t)|
			\cr
		&\le
			\sum_{\lambda=l+1}^{r} (2^d-1) 2^{-k_\lambda}
			+ (2^d-1) 2^{-k_l}
			\cr
		&\quad
			+ \sum_{\lambda=l+1}^{r} 2^{-k_\lambda-1}
			+ \sum_{\lambda=l+1}^{r} \sqrt{d} 2^{-k_{\lambda-1}}
			+ \sqrt{d} 2^{-k_l-1}
			\cr
		&\le c(d)\, 2^{-k_l}
\end{align*}
and
\begin{equation*}
	\dist\bigl(\gamma(t),\partial\Oij\bigr) \ge 2^{-k_l-2}.
\end{equation*}
Therefore, by the definition of the John domain, $C_J(\Oij)\le 4c$.
\end{proof}

\begin{proof}[Proof of Lemma~\ref{lemm:WBern}]
It is enough to show that for every
$n$-dimensional subspace $E$ of $V_0^1L^{d,1}(Q)$
satisfying
\begin{equation} \label{lefdas}
	\|u\|_\infty \ge \varrho \|\nabla u\|_{d,1}
		\quad\text{for $u\in E$},
\end{equation}
there is the inequality
\begin{equation} \label{lefdst}
	\varrho \le c n^{-\frac{1}{d}},
\end{equation}
where $c=c(d)$ is the absolute constant.

We proceed similarly as in the proof of Lemma~\ref{lemm:VBern}.
We may assume that the all functions in $E$ are continuous and, since
$E$ is of finite dimension, they are uniformly continuous.
Formally written, to any given $\varepsilon>0$,
there is a $\delta>0$ such that
\begin{equation*}
	|u(x) - u(y)|\le \varepsilon \|u\|_\infty
		\quad\text{whenever $|x-y|<\delta$ and $u\in E$}.
\end{equation*}
Let $\mathcal Q_k$ be the family of dyadic cubes as in \eqref{dicubes}
where $k\in\N$ is chosen sufficiently large to ensure that
the cubes in $\mathcal Q_k$ have diameter less than $\delta$.
Denote $N=2^{dk}$, the total number of the cubes in $\mathcal Q_k$.
We may also assume that $N>n$.
We label the cubes along the $k$-th approximation of Hilbert
curve as is in Lemma~\ref{lemm:John} by the labeling map
$Q\colon\{1,2,\dots,N\}\to\mathcal Q_k$. The center point of
a cube $Q_l$ will be denoted by $x_l$.

Let us define the mapping $\Phi\colon E\to \ell^\infty_N$ by
\begin{equation*}
	(\Phi u)_l = u(x_l)
		\quad\text{for $1\le l\le N$}.
\end{equation*}
Function $\Phi$ is well defined, since $E$ consists of continuous
functions.  Clearly $\|\Phi u\|_\infty \le \|u\|_\infty$ and also
$(1-\varepsilon) \|u\|_{\infty} \le \|\Phi u\|_\infty$ by the same
procedure as in \eqref{Ephibound} of Proposition~\ref{prop:CLL}. Thus
$\Phi$ is an isomorphism from $E$ onto a $n$-dimensional subspace of
$\ell^\infty_N$. On applying Zigzag theorem \citep[Theorem 1]{Cha:09},
there is an element $v\in E$ with $\|\Phi v\|_\infty=1$ and a
subsequence of $\{1,2,\dots,N\}$ of length $n$ satisfying
\begin{equation*}
	v(x_{l_j}) = (-1)^j
		\quad\text{for $1\le j\le n$}.
\end{equation*}
Denote $\Omega_j = \Omega_{l_j,l_{j+1}}$ using the definition of
the latter symbol given in \eqref{Oij}.
We have
\begin{align}
	\begin{split} \label{apr1}
	2(n-1) &= \sum_{j=1}^{n-1} | v(x_{l_j}) - v(x_{l_{j+1}}) |
			\cr
		&\le \sum_{j=1}^{n-1} C_J(\Omega_j) \|\nabla u \chi_{\Omega_j}\|_{d,1}
			\cr
		&\le
			\biggl(\,
			\sum_{j=1}^{n-1} C_J(\Omega_j)^{d'}
			\biggr)^{\frac{1}{d'}}
			\biggl(\,
				\sum_{j=1}^{n-1} \|\nabla v \chi_{\Omega_j}\|_{d,1}^d
			\biggr)^{\frac{1}{d}}
			\cr
		&\le C (n-1)^\frac{1}{d'} \|\nabla v\|_{d,1},
		\end{split}
\end{align}
where, in the first term, we used Lemma~\ref{lemm:John}
and, in the second term, we used the summation property of
the Lorentz norm \eqref{Lorsum}. Note that the sets $\Omega_j$ are not disjoint,
however $\sum_j \chi_{\Omega_j} \le 2$ on $Q$ which only affects
the constant $C$.

Now, combining \eqref{apr1} and \eqref{lefdas}, we get
\begin{equation*}
	2(n-1)^{\frac{1}{d}}
		\le C \|\nabla v\|_{d,1}
		\le \frac{C}{\varrho}\, \|v\|_{\infty}
		\le \frac{C}{\varrho(1-\varepsilon)}
\end{equation*}
and the inequality \eqref{lefdst} follows by letting $\varepsilon\to 0^+$.
\end{proof}

\begin{proof}[Proof of Theorem~\ref{thm:ddim}]
The proof of \eqref{mb4} follows from Lemma~\ref{lemm:Wiso}
and Lemma~\ref{lemm:WBern}.
The relation \eqref{mb5} is a consequence of the non-compactness
of the embedding in question.
\end{proof}

\begin{proof}[Proof of Corollary~\ref{cor:WXC}]
Let $s_n$ represent Isomorphism or Bernstein numbers. We have
\begin{equation*}
	 cn^{-1/d}
	 	\le i_n(V_0^1X \embed\C)
	 	\le s_n(V_0^1X\embed C)
		\le s_n(V_0^1L^{d,1}\embed\C)
		\le cn^{-1/d},
\end{equation*}
where the first inequality follows by Lemma~\ref{lemm:Wiso},
the second by the ideal property (S3) of $s$-numbers
and the last one is due to Theorem~\ref{thm:ddim}.
\end{proof}

\newcommand{\etalchar}[1]{$^{#1}$}


\begin{thebibliography}{FHH{\etalchar{+}}11}

\bibitem[AN00]{Alb:00}
J.~Alber and R.~Niedermeier.
\newblock On multidimensional curves with {H}ilbert property.
\newblock {\em Theory Comput. Syst.}, 33(4):295--312, 2000.

\bibitem[Boj88]{Boj:88}
B.~Bojarski.
\newblock Remarks on {S}obolev imbedding inequalities.
\newblock In {\em Complex analysis, {J}oensuu 1987}, volume 1351 of {\em
  Lecture Notes in Math.}, pages 52--68. Springer, Berlin, 1988.

\bibitem[BS88]{Ben:88}
C.~Bennett and R.~Sharpley.
\newblock {\em Interpolation of operators}, volume 129 of {\em Pure and Applied
  Mathematics}.
\newblock Academic Press, Inc., Boston, MA, 1988.

\bibitem[But69]{But:69}
A.~R. Butz.
\newblock Convergence with {H}ilbert's space filling curve.
\newblock {\em J. Comput. System Sci.}, 3:128--146, 1969.

\bibitem[CFP{\etalchar{+}}09]{Cha:09}
I.~Chalendar, E.~Fricain, A.~I. Popov, D.~Timotin, and V.~G. Troitsky.
\newblock Finitely strictly singular operators between {J}ames spaces.
\newblock {\em J. Funct. Anal.}, 256(4):1258--1268, 2009.

\bibitem[CP98]{Cia:98a}
A.~Cianchi and L.~Pick.
\newblock Sobolev embeddings into {BMO}, {VMO}, and {$L_\infty$}.
\newblock {\em Ark. Mat.}, 36(2):317--340, 1998.

\bibitem[CS90]{Car:90}
B.~Carl and I.~Stephani.
\newblock {\em Entropy, compactness and the approximation of operators},
  volume~98 of {\em Cambridge Tracts in Mathematics}.
\newblock Cambridge University Press, Cambridge, 1990.

\bibitem[FHH{\etalchar{+}}11]{Fab:11}
M.~Fabian, P.~Habala, P.~H{\'a}jek, V.~Montesinos, and V.~Zizler.
\newblock {\em Banach space theory}.
\newblock CMS Books in Mathematics/Ouvrages de Math\'ematiques de la SMC.
  Springer, New York, 2011.
\newblock The basis for linear and nonlinear analysis.

\bibitem[LE11]{Lan:11}
J.~Lang and D.~Edmunds.
\newblock {\em Eigenvalues, embeddings and generalised trigonometric
  functions}, volume 2016 of {\em Lecture Notes in Mathematics}.
\newblock Springer, Heidelberg, 2011.

\bibitem[Mar88]{Mar:88}
O.~Martio.
\newblock John domains, bi-{L}ipschitz balls and {P}oincar\'e inequality.
\newblock {\em Rev. Roumaine Math. Pures Appl.}, 33(1-2):107--112, 1988.

\bibitem[MJFS01]{Moo:01}
B.~Moon, H.~V. Jagadish, C.~Faloutsos, and J.~H. Saltz.
\newblock Analysis of the clustering properties of the hilbert space-filling
  curve.
\newblock {\em IEEE Transactions on Knowledge and Data Engineering},
  13(1):124--141, 2001.

\bibitem[Ngu15]{Ngu:15}
V.~K. Nguyen.
\newblock Bernstein numbers of embeddings of isotropic and dominating mixed
  {B}esov spaces.
\newblock {\em Math. Nachr.}, 288(14-15):1694--1717, 2015.

\bibitem[PB89]{Pla:89}
L.~K. Platzman and J.~J. Bartholdi, III.
\newblock Spacefilling curves and the planar travelling salesman problem.
\newblock {\em J. Assoc. Comput. Mach.}, 36(4):719--737, 1989.

\bibitem[Pie74]{Pie:74}
A.~Pietsch.
\newblock {$s$}-numbers of operators in {B}anach spaces.
\newblock {\em Studia Math.}, 51:201--223, 1974.

\bibitem[Pie07]{Pie:07}
A.~Pietsch.
\newblock {\em History of {B}anach spaces and linear operators}.
\newblock Birkh\"auser Boston, Inc., Boston, MA, 2007.

\bibitem[Pin85]{Pin:85}
A.~Pinkus.
\newblock {\em {$n$}-widths in approximation theory}, volume~7 of {\em
  Ergebnisse der Mathematik und ihrer Grenzgebiete (3) [Results in Mathematics
  and Related Areas (3)]}.
\newblock Springer-Verlag, Berlin, 1985.

\bibitem[Res80]{Res:80}
Yu.~G. Reshetnyak.
\newblock Integral representations of differentiable functions in domains with
  a nonsmooth boundary.
\newblock {\em Sibirsk. Mat. Zh.}, 21(6):108--116, 221, 1980.

\end{thebibliography}
\end{document}